\documentclass[11pt, a4paper, reqno]{amsart}
\author{Soichiro Fujii}
\thanks{The author is supported by ERATO HASUO Metamathematics for Systems Design Project (No.~JPMJER1603), JST}
\date{\today}
\address{Research Institute for Mathematical Sciences, Kyoto University \\ 
Kyoto 606-8502, Japan}
\title{Introduction to universal algebra and clones}
\email{s.fujii.math@gmail.com}
\usepackage[centering, totalwidth=380pt, totalheight=600pt]{geometry}
\usepackage[pagebackref=true, colorlinks]{hyperref}
\usepackage{amsfonts, amssymb, amsthm}
\usepackage{comment}
\usepackage{bussproofs}
\usepackage{thmtools}
\usepackage{tikz}
\usetikzlibrary{arrows.meta,decorations.pathreplacing,decorations.markings,shapes,calc}
\usetikzlibrary{matrix,arrows,decorations.pathmorphing,backgrounds,decorations.markings,positioning}
\tikzset{2cell/.style={-implies,double,double equal sign distance,shorten 
>=2pt, shorten <=3pt}}
\tikzset{2cellshort/.style={-implies,double,double equal sign distance,shorten 
>=4pt, shorten <=5pt}}
\tikzset{2cellr/.style={implies-,double,double equal sign distance,shorten 
>=3pt, shorten <=2pt}}
\tikzset{3cell/.style={-implies,double,double distance=2.5pt,shorten >=2pt, 
shorten <=3pt}}
\tikzset{labelsize/.style={font=\scriptsize}}
\tikzset{string/.style={very thick}}
\tikzset{
  pto/.style={->,postaction={decorate},
    decoration={
        markings,
        mark=at position 0.5 with {\arrow{|}}}
  },
}

\declaretheorem[style=plain,numberwithin=section,name=Theorem]{thm}

\declaretheorem[style=plain,sibling=thm,name=Proposition]{proposition}

\declaretheorem[style=definition,qed=$\blacksquare$,sibling=thm,name=Definition]{definition}

\mathchardef\mhyphen="2D

\newcommand{\Set}{\mathbf{Set}}

\newcommand{\monoid}[1]{\mathsf{#1}}
\newcommand{\End}{\monoid{End}}
\newcommand{\Endcl}[1]{\monoid{End}(#1)}

\newcommand{\enrich}[2]{{\langle #1, #2\rangle}}

\newcommand{\F}{\mathbf{F}}

\newcommand{\NN}{\mathbb{N}}
\newcommand{\pres}[2]{{{\langle\, #1\, |\, #2 \,\rangle}}}
\newcommand{\interp}[1]{{[\![#1]\!]}}
\newcommand{\interpp}[1]{{[\![#1]\!]'}}
\newcommand{\group}{\mathrm{Grp}}

\newcommand{\clo}{\mathrm{Clo}}

\makeatletter
\DeclareRobustCommand{\rvdots}{%
  \vbox{
    \baselineskip4\p@\lineskiplimit\z@
    \kern-\p@
    \hbox{.}\hbox{.}\hbox{.}
  }}
\makeatother

\makeatletter
\newcommand{\pto}{}
\newcommand{\pgets}{}
\DeclareRobustCommand{\pto}{\mathrel{\mathpalette\p@to@gets\to}}
\DeclareRobustCommand{\pgets}{\mathrel{\mathpalette\p@to@gets\gets}}
\newcommand{\p@to@gets}[2]{%
  \ooalign{\hidewidth$\m@th#1\mapstochar\mkern5mu$\hidewidth\cr$\m@th#1\longrightarrow$\cr}%
}
\makeatother

\newcommand{\op}{{\mathrm{op}}}

\newcommand{\defemph}[1]{\textbf{#1}}

\newcommand{\str}[1]{\mathbf{#1}}

\newcommand{\law}[1]{\mathbb{#1}}

\begin{document}
\begin{abstract}
The purpose of this note is to provide a gentle introduction to basic universal 
algebra and (abstract) clones.
\end{abstract}
\maketitle
\section{Introduction}
In almost every field of pure and applied mathematics, 
\emph{algebras} (in a broad sense) arise quite naturally 
in one way or another.
An algebra, typically, is a set equipped with 
a family of operations on it.
So for example the symmetric group of degree five $\mathfrak{S}_5$
and the ring of integers $\mathbb{Z}$ are both algebras.
Structural similarities between important algebras have led to 
the introduction and study of various \emph{types of algebras},
such as 
monoids, groups, rings, vector spaces over a field, lattices, 
Boolean algebras, and Heyting algebras.
A type of algebras is normally specified by a family of operations
and a family of equational axioms. 
We shall call such a specification of a type of algebras an \emph{algebraic theory}.

Subsequently, 
various authors have set out to develop a background theory, or a  {metatheory} for a certain type of algebraic theories.
The most famous classical example is Birkhoff's \emph{universal algebra}~\cite{Birkhoff_abst_alg}.\footnote{Other examples of metatheories include those of non-symmetric and symmetric operads \cite{May_loop}, PROs and PROPs \cite{MacLane_cat_alg}, generalised operads 
\cite{Burroni_T_cats,Kelly_club_data_type,Hermida_representable,Leinster_book}, and monads \cite{Eilenberg_Moore,Linton_equational}; cf.~\cite{Fujii_thesis,Fujii_unified}.}
By working at this level of generality, one can prove theorems for various types of algebras
once and for all; for instance, the homomorphism theorems
in universal algebra (see e.g., \cite[Section II.6]{Burris_Sankappanavar}) 
generalise the homomorphism theorems for groups
to monoids, rings, lattices, etc.
A metatheory also provides a method to relate different types of algebras, by means of morphisms between algebraic theories.

In this note, we explain the basics of universal algebra. 
We shall confine ourselves to the most basic definitions; we focus on \emph{presentations of equational theories}, the type of algebraic theories universal algebra deals with. 
We then describe a presentation independent version of them, namely \emph{(abstract) clones}.
This note is based on Sections 2.1 and 2.2 of the author's thesis \cite{Fujii_thesis}.

\section{Universal algebra}
\label{sec:univ_alg}
Universal algebra~\cite{Birkhoff_abst_alg} deals with types of algebras
defined by finitary operations and equations between them.
As a running example, let us consider 
\emph{groups}.
A group may be defined as a set $G$
equipped with an element $e^G\in G$ (the unit), and 
two functions $i^G\colon G\longrightarrow G$ (the inverse) and 
$m^G\colon G\times G\longrightarrow G$ (the multiplication),
satisfying the following axioms:
\begin{itemize}
\item for all $g_1\in G$, $m^G(g_1,e^G)=g_1$ (the right unit axiom);
\item for all $g_1\in G$, $m^G(g_1,i^G(g_1))=e^G$ (the right inverse axiom);
\item for all $g_1,g_2,g_3\in G$, $m^G(m^G(g_1,g_2),g_3)=m^G(g_1,m^G(g_2,g_3))$ 
(the associativity axiom).\footnote{From these three axioms it follows that 
for all $g_1\in G$, $m^G(e^G,g_1)=g_1$ (the left unit axiom) and 
$m^G(i^G(g_1),g_1)=e^G$ (the left inverse axiom) hold.}
\end{itemize}
This definition of group turns out to be an instance of 
the notion of \emph{presentation of an equational theory},
one of the most fundamental notions in universal algebra.

\medskip

First we introduce the notion of \emph{graded set}, which provides a convenient language for our exposition.
\begin{definition}
\label{def:graded_set}
\begin{enumerate}
\item An \defemph{($\NN$-)graded set} $\Gamma$ is a family 
$\Gamma=(\Gamma_n)_{n\in\NN}$ of sets
indexed by natural numbers $\NN=\{0,1,2,\dots\}$.
By an \defemph{element of $\Gamma$} we mean an element of the 
set $\coprod_{n\in\NN} \Gamma_n=\{\,(n,\gamma)\mid n\in\NN, \gamma\in 
\Gamma_n\,\}$.
We write $x\in\Gamma$ iff $x$ is an element of $\Gamma$.
\item If $\Gamma=(\Gamma_n)_{n\in\NN}$ and $\Gamma'=(\Gamma'_n)_{n\in\NN}$
are graded sets, then a \defemph{morphism of graded sets} $f\colon 
\Gamma\longrightarrow
\Gamma'$ is a family of functions $f=(f_n\colon \Gamma_n\longrightarrow 
\Gamma'_n)_{n\in\NN}$.\qedhere
\end{enumerate}
\end{definition}

We can routinely extend the basic notions of set theory to graded sets.
For example, we say that a graded set $\Gamma'$ is a \defemph{graded subset}
of a graded set $\Gamma$ (written as $\Gamma'\subseteq \Gamma$)
iff for each $n\in\NN$, $\Gamma'_n$ is a subset of $\Gamma_n$.
Given arbitrary graded sets $\Gamma$ and $\Gamma'$, their 
\defemph{cartesian product} (written as $\Gamma\times\Gamma'$) is defined by 
$(\Gamma\times\Gamma')_n=\Gamma_n\times\Gamma'_n$ for each $n\in\NN$.
An \defemph{equivalence relation} on a graded set $\Gamma$ is 
a graded subset $R\subseteq \Gamma\times\Gamma$ such that each $R_n\subseteq
\Gamma_n\times\Gamma_n$ is an equivalence relation on the set $\Gamma_n$.
Given such an equivalence relation $R$ on $\Gamma$, we can form the 
\defemph{quotient graded set} $\Gamma/R$ by 
setting $(\Gamma/R)_n=\Gamma_n/R_n$, the quotient set of $\Gamma_n$
with respect to $R_n$.
These notions will be used below.

A graded set can be seen as a \emph{(functional) signature}.
That is, we can regard a graded set $\Sigma$ as the signature whose set of 
$n$-ary functional symbols is given by $\Sigma_n$ for each $n\in\NN$.
We often use the symbol $\Sigma$ to denote a graded set when we want to
emphasise this aspect of graded sets, as in the following definition. 

\begin{definition}
\label{def:Sigma_alg}
Let $\Sigma$ be a graded set.
\begin{enumerate}
\item A \defemph{$\Sigma$-algebra} $\str{A}$ is a set $A$ equipped with, for each $n\in\NN$ and $\sigma\in\Sigma_n$,
a function $\interp{\sigma}^\str{A}\colon A^n\longrightarrow A$
called the \defemph{interpretation of $\sigma$}.\footnote{Note that we allow the set $A$ to be empty. In traditional universal algebra the underlying set of a $\Sigma$-algebra is usually required to be nonempty.}
We write such a $\Sigma$-algebra $\str{A}=(A,(\interp{\sigma}^\str{A})_{n\in\NN,\sigma\in\Sigma_n})$
simply as $(A,\interp{-}^\str{A})$.
We sometimes omit the superscript in $\interp{-}^\str{A}$.
\item If $\str{A}=(A,\interp{-}^\str{A})$
and $\str{B}=(B,\interp{-}^{\str{B}})$
are $\Sigma$-algebras,
then a \defemph{$\Sigma$-homomorphism from 
$\str{A}$ to 
$\str{B}$}
is a function $f\colon A\longrightarrow B$ such that for any $n\in\NN$,
$\sigma\in\Sigma_n$ and $a_1,\dots,a_n\in A$,
\[
f(\interp{\sigma}^\str{A}(a_1,\dots,a_n))=\interp{\sigma}^\str{B}(f(a_1),\dots,f(a_n))
\]
holds (that is, the diagram
\[
\begin{tikzpicture}[baseline=-\the\dimexpr\fontdimen22\textfont2\relax ]
      \node (TL) at (0,2)  {$A^n$};
      \node (TR) at (3,2)  {$B^n$};
      \node (BL) at (0,0) {$A$};
      \node (BR) at (3,0) {$B$};
      \draw[->] (TL) to node[auto,labelsize](T) {$f^n$} (TR);
      \draw[->]  (TR) to node[auto,labelsize] {$\interp{\sigma}^\str{B}$} (BR);
      \draw[->]  (TL) to node[auto,swap,labelsize] {$\interp{\sigma}^\str{A}$} (BL);
      \draw[->] (BL) to node[auto,labelsize](B) {$f$} (BR);
\end{tikzpicture} 
\]
commutes).\qedhere
\end{enumerate}
\end{definition}

As an example, let us consider the graded set $\Sigma^\group$
defined as $\Sigma^\group_0=\{e\}$, $\Sigma^\group_1=\{i\}$,
$\Sigma^\group_2=\{m\}$ and $\Sigma^\group_n=\emptyset$ for all $n\geq 3$.
Then the structure of a group is given by
that of a $\Sigma^\group$-algebra.
Note that to give an element $e^G\in G$ is equivalent to give a 
function $\interp{e}\colon 1\longrightarrow G$ where $1$ is 
a singleton set, and that for any set $G$, $G^0$ is a singleton set.
Also, between groups, the notions of group homomorphism and 
$\Sigma^\group$-homomorphism coincide.

However, not all $\Sigma^\group$-algebras are groups;
for a $\Sigma^\group$-algebra to be a group, the interpretations
must satisfy the group axioms.
Notice that all group axioms are {equations} between
certain expressions built from
variables and operations.
This is the fundamental feature shared by all types of algebras
expressible in universal algebra.
The following notion of $\Sigma$-term defines {``expressions
built from variables and operations''}
relative to arbitrary graded sets $\Sigma$.

\begin{definition}
\label{def:Sigma_term}
Let $\Sigma$ be a graded set.
The graded set $T(\Sigma)=(T(\Sigma)_n)_{n\in\NN}$ of \defemph{$\Sigma$-terms}
is defined inductively as follows.
\begin{enumerate}
\item For each $n\in\NN$ and  $i\in\{1,\dots,n\}$, 
\[
x_i^{(n)}\in T(\Sigma)_n.
\]
We sometimes omit the superscript and write $x_i$ for $x_i^{(n)}$.
\item For each $n,k\in\NN$, $\sigma\in\Sigma_k$ and $t_1,\dots,t_k\in 
T(\Sigma)_n$,
\[
\sigma(t_1,\dots,t_k)\in T(\Sigma)_n.
\]
When $k=0$, we usually omit the parentheses in $\sigma()$
and write instead as $\sigma$.\qedhere
\end{enumerate}
\end{definition}

An immediate application of the inductive nature of the above
definition of $\Sigma$-terms is the canonical extension of the interpretation 
function $\interp{-}$
of a $\Sigma$-algebra from $\Sigma$ to $T(\Sigma)$.

\begin{definition}
\label{def:interp_Sigma_term}
Let $\Sigma$ be a graded set and 
$\str{A}=(A,\interp{-}^\str{A})$ be a 
$\Sigma$-algebra.
We define the \defemph{interpretation} $\interpp{-}^\str{A}$ of $\Sigma$-terms
recursively as follows.
\begin{enumerate}
\item For each $n\in\NN$ and $i\in\{1,\dots,n\}$,
\[
\interpp{x^{(n)}_i}^\str{A}\colon A^n\longrightarrow A
\]
is the $i$-th projection $(a_1,\dots,a_n)\longmapsto a_i$.
\item For each $n,k\in\NN$, $\sigma\in\Sigma_k$ and $t_1\dots,t_k\in 
T(\Sigma)_n$,
\[
\interpp{\sigma(t_1,\dots,t_k)}^\str{A}\colon A^n\longrightarrow A
\] 
maps $(a_1,\dots,a_n)\in A^n$ to 
$\interp{\sigma}^\str{A}(\interpp{t_1}^\str{A}(a_1,\dots,a_n),\dots,
\interpp{t_k}^\str{A}(a_1,\dots,a_n))$;
that is, the function $\interpp{\sigma(t_1,\dots,t_k)}^\str{A}$ is the following composite:
\[
\begin{tikzpicture}[baseline=-\the\dimexpr\fontdimen22\textfont2\relax ]
      \node (1) at (0,0)  {$A^n$};
      \node (2) at (3,0)  {$A^k$};
      \node (3) at (5,0) {$A.$};
      \draw[->] (1) to node[auto,labelsize]{$\langle 
      \interpp{t_1}^\str{A},\dots,\interpp{t_k}^\str{A}\rangle$} (2);
      \draw[->] (2) to node[auto,labelsize] {$\interp{\sigma}^\str{A}$} (3);
\end{tikzpicture} 
\]
\end{enumerate}
Note that for any $n\in\NN$ and $\sigma\in\Sigma_n$,
$\interp{\sigma}^\str{A}=\interpp{\sigma(x^{(n)}_1,\dots,x^{(n)}_n)}^\str{A}$.
Henceforth, for any $\Sigma$-term $t$
we simply write $\interp{t}^\str{A}$ for $\interpp{t}^\str{A}$ defined above.
\end{definition}

\begin{definition}
\label{def:Sigma_eq}
Let $\Sigma$ be a graded set.
An element of the graded set $T(\Sigma)\times T(\Sigma)$ is called a 
\defemph{$\Sigma$-equation}.
We write a $\Sigma$-equation $(n,(t,s))\in T(\Sigma)\times T(\Sigma)$ (that is,
$n\in\NN$ and $t,s\in T(\Sigma)_n$) as $t\approx_n s$ or $t\approx s$.
\end{definition}

\begin{definition}
\label{defn:univ_alg_pres_eq_thy}
A \defemph{presentation of an equational theory} $\pres{\Sigma}{E}$ 
is a pair consisting of:
\begin{itemize}
\item a graded set $\Sigma$ of \defemph{basic operations}, 
and
\item a graded set $E\subseteq T(\Sigma)\times T(\Sigma)$
of \defemph{equational axioms}.\qedhere
\end{itemize}
\end{definition}

\begin{definition}
\label{defn:univ_alg_model}
Let $\pres{\Sigma}{E}$ be a presentation of an equational theory.
\begin{enumerate}
\item A \defemph{model of $\pres{\Sigma}{E}$}, or a \defemph{$\pres{\Sigma}{E}$-model}, is a $\Sigma$-algebra
$\str{A}$
such that for any $t\approx_n s\in E$,
$\interp{t}^\str{A}=\interp{s}^\str{A}$ holds.
\item A \defemph{homomorphism}
between models of $\pres{\Sigma}{E}$ is just
a $\Sigma$-homomorphism between the corresponding
$\Sigma$-algebras.\qedhere
\end{enumerate}
\end{definition}

Consider the presentation of an equational theory 
$\pres{\Sigma^\group}{E^\group}$, where 
\[E^\group_1=\{\,m(x^{(1)}_1,e)\approx x^{(1)}_1, \quad
m(x^{(1)}_1,i(x^{(1)}_1))\approx e\,\},
\]
\[ E^\group_3=\{\,m(m(x^{(3)}_1,x^{(3)}_2),x^{(3)}_3)\approx
m(x^{(3)}_1,m(x^{(3)}_2,x^{(3)}_3))\,\}\]
and $E^\group_n=\emptyset$ for all $n\in\NN\setminus \{1,3\}$.
Clearly, groups are the same as models of
$\pres{\Sigma^\group}{E^\group}$.
Many other types of algebras---indeed all examples we have mentioned in the 
first paragraph of the introduction---can be written as models of $\pres{\Sigma}{E}$
for a suitable choice of the presentation of an equational theory
$\pres{\Sigma}{E}$ 
(see e.g.,~\cite{Burris_Sankappanavar}).

\medskip

We now describe the machinery of 
\emph{equational logic}, which enables us to investigate consequences of 
equational axioms without referring to their models.
We assume that the reader is familiar with the basics of 
mathematical logic, such as 
substitution of a term $t$ for a variable $x$ in a term $s$ 
(written as $s[x\mapsto t]$), simultaneous substitution
(written as $s[x_1\mapsto t_1, \dots, x_k\mapsto t_k]$), and 
the notion of proof (tree) and its definition by inference rules.

\begin{definition}
\label{def:eq_logic}
Let $\pres{\Sigma}{E}$ be a presentation of an equational theory.
\begin{enumerate}
\item Define the set of \defemph{$\pres{\Sigma}{E}$-proofs} inductively
by the following inference rules.
Every $\pres{\Sigma}{E}$-proof is a finite rooted tree whose vertices 
are labelled by $\Sigma$-equations.
\begin{center}
\AxiomC{}
\LeftLabel{(\sc{Ax})\ }
\RightLabel{\ (if $t\approx_n s \in E$)}
\UnaryInfC{$t\approx_n s$}
\DisplayProof
\medskip

\bottomAlignProof
\AxiomC{}
\LeftLabel{(\sc{Refl})\ }
\UnaryInfC{$t\approx_n t$}
\DisplayProof
\quad
\bottomAlignProof
\AxiomC{$t\approx_n s $}
\LeftLabel{(\sc{Sym})\ }
\UnaryInfC{$s\approx_n t$}
\DisplayProof
\quad
\bottomAlignProof
\AxiomC{$t\approx_n s$}
\AxiomC{$s\approx_n u$}
\LeftLabel{(\sc{Trans})\ }
\BinaryInfC{$t\approx_n u$}
\DisplayProof
\medskip

\bottomAlignProof
\AxiomC{$s\approx_k s'$}
\AxiomC{$t_1\approx_n t'_1$}
\AxiomC{$\cdots$}
\AxiomC{$t_k\approx_n t'_k$}
\LeftLabel{(\sc{Cong})\ }
\QuaternaryInfC{$s[x^{(k)}_1 \mapsto t_1, \dots, x^{(k)}_k\mapsto t_k]\approx_n 
s'[x^{(k)}_1 \mapsto t'_1, \dots, x^{(k)}_k\mapsto t'_k]$}
\DisplayProof
\end{center}
\item A $\Sigma$-equation 
$t\approx_n s\in T(\Sigma)\times T(\Sigma)$ is called an
\defemph{equational theorem of $\pres{\Sigma}{E}$} iff there exists 
a $\pres{\Sigma}{E}$-proof whose root is labelled by $t\approx_n s$.
We write \[
\pres{\Sigma}{E}\vdash t\approx_n s\]
to mean that $t\approx_n s$ is 
an equational theorem of $\pres{\Sigma}{E}$, and denote by 
$\overline{E}\subseteq T(\Sigma)\times T(\Sigma)$ the 
graded set of all equational theorems of $\pres{\Sigma}{E}$.
\qedhere
\end{enumerate}
\end{definition}

The assertion $\pres{\Sigma}{E}\vdash t\approx s$ says that the $\Sigma$-equation $t\approx s$ is a \emph{syntactic} consequence of the equational axioms ${E}$. 
Its counterpart is the \emph{semantic} consequence relation $\vDash$, defined as follows.

\begin{definition}
\label{def:semantical_consequence_rel}
\begin{enumerate}
\item Let $\Sigma$ be a graded set and 
$\str{A}$ be a $\Sigma$-algebra.
For any $\Sigma$-equation $t\approx_n s\in T(\Sigma)\times T(\Sigma)$, we write 
\[
\str{A}\vDash t\approx_n s
\]
to mean $\interp{t}^\str{A}=\interp{s}^\str{A}$.
\item Let $\pres{\Sigma}{E}$ be a presentation of an equational theory.
For any $\Sigma$-equation $t\approx_n s\in T(\Sigma)\times T(\Sigma)$, we write
\[
\pres{\Sigma}{E}\vDash t\approx_n s
\] 
to mean that for any $\pres{\Sigma}{E}$-model $\str{A}$,
$\str{A}\vDash t\approx_n s$.\qedhere
\end{enumerate}
\end{definition}

Equational logic is known to be both \emph{sound} and \emph{complete}, meaning that the two relations $\vdash$ and $\vDash$ coincide. 

\begin{thm}
\label{thm:eq_logic_sound_complete}
Let $\pres{\Sigma}{E}$ be a presentation of an equational theory.
\begin{enumerate}
\item (Soundness) Let $t\approx_n s\in T(\Sigma)\times T(\Sigma)$.
If $\pres{\Sigma}{E}\vdash t\approx_n s$ then $\pres{\Sigma}{E}\vDash 
t\approx_n s$.
\item (Completeness) Let $t\approx_n s\in T(\Sigma)\times T(\Sigma)$.
If $\pres{\Sigma}{E}\vDash t\approx_n s$ then $\pres{\Sigma}{E}\vdash 
t\approx_n s$.
\end{enumerate}
\end{thm}
\begin{proof}
The soundness theorem can be shown by a straightforward induction over
$\pres{\Sigma}{E}$-proofs.

To prove the completeness theorem, first observe that 
the graded set $\overline{E}\subseteq T(\Sigma)\times T(\Sigma)$ of all equational theorems of $\pres{\Sigma}{E}$ (Definition~\ref{def:eq_logic}) is an equivalece relation on $T(\Sigma)$, thanks to the rules ({\sc Refl}), ({\sc Sym}) and ({\sc Trans}).
Hence we can consider the quotient graded set $T(\Sigma)/\overline{E}$.
We claim that for each $n\in\NN$, the set $T^{\pres{\Sigma}{E}}_n=(T(\Sigma)/\overline{E})_n$ has a natural structure of $\pres{\Sigma}{E}$-model.

We start with endowing a $\Sigma$-algebra structure on the set $T^\pres{\Sigma}{E}_n$; that is, we define for each $k\in\NN$ and each $\sigma\in\Sigma_k$, its interpretation $\interp{\sigma}\colon (T^\pres{\Sigma}{E}_n)^k\longrightarrow T^\pres{\Sigma}{E}_n$.
This is defined as 
\begin{equation*} 
	\interp{\sigma}([t_1]_{\overline{E}},\dots,[t_k]_{\overline{E}})=[\sigma(t_1,\dots,t_k)]_{\overline{E}}
\end{equation*}
for each $t_1,\dots, t_k\in T(\Sigma)_n$.
To see that it is indeed well-defined, consider the instances of the ({\sc{Cong}}) rule where $s=s'=\sigma(x^{(k)}_1,\dots,x^{(k)}_k)$. 
Observe that in this $\Sigma$-algebra, the interpretation of a $\Sigma$-term $s\in T(\Sigma)_k$ is given by 
\[
\interp{s}([t_1]_{\overline{E}},\dots,[t_k]_{\overline{E}})
=[s[x^{(k)}_1\mapsto t_1,\dots, x^{(k)}_k\mapsto t_k] ]_{\overline{E}}.
\]

The $\Sigma$-algebra $\str{T}^\pres{\Sigma}{E}_n=(T^\pres{\Sigma}{E}_n,\interp{-})$ satisfies all equational axioms of $\pres{\Sigma}{E}$. To see this, notice that if $s\approx_k s'\in E$, then for each $t_1,\dots,t_k\in T(\Sigma)_n$, the $\Sigma$-equation $s[x^{(k)}_1\mapsto t_1,\dots, x^{(k)}_k\mapsto t_k] \approx_n s'[x^{(k)}_1\mapsto t_1,\dots, x^{(k)}_k\mapsto t_k]$ is an equational theorem of $\pres{\Sigma}{E}$, by the rules ({\sc Ax}), ({\sc Refl}) and ({\sc Cong}). Hence $\interp{s}=\interp{s'}$ holds in $\str{T}^\pres{\Sigma}{E}_n$.

Now suppose that for a $\Sigma$-equation $t\approx_n s$ we have $\pres{\Sigma}{E}\vDash t\approx_n s$. Then in particular $\str{T}^\pres{\Sigma}{E}_n\vDash t\approx_n s$, and in particular the images of $[x^{(n)}_1]_{\overline{E}},\dots, [x^{(n)}_n]_{\overline{E}}\in T^\pres{\Sigma}{E}_n$ under the functions $\interp{t}$ and $\interp{s}$ agree.
Hence we have
\begin{align*}
[t]_{\overline{E}} &= [t[x^{(n)}_1\mapsto x^{(n)}_1, \dots, x^{(n)}_n\mapsto x^{(n)}_n]]_{\overline{E}}\\
&= \interp{t}([x^{(n)}_1]_{\overline{E}},\dots, [x^{(n)}_n]_{\overline{E}})\\
&= \interp{s}([x^{(n)}_1]_{\overline{E}},\dots, [x^{(n)}_n]_{\overline{E}})\\
&= [s[x^{(n)}_1\mapsto x^{(n)}_1, \dots, x^{(n)}_n\mapsto x^{(n)}_n]]_{\overline{E}}\\
&=[s]_{\overline{E}},
\end{align*}
namely
$\pres{\Sigma}{E}\vdash t\approx_n s$.
\end{proof}

\medskip

Before closing this section, 
we remark that the $\pres{\Sigma}{E}$-model $\str{T}^\pres{\Sigma}{E}_n$ used in the above proof
is in fact the \emph{free $\pres{\Sigma}{E}$-model generated by the $n$-element set  $X_n=\{x^{(n)}_1,\dots,x^{(n)}_n\}$}, in the following sense.
\begin{proposition}
\label{prop:free_model}
Let $\pres{\Sigma}{E}$ be a presentation of an equational theory and $n$ be a natural number. Define the function $\eta_{X_n}\colon X_n\longrightarrow T^\pres{\Sigma}{E}_n$ by 
$\eta_{X_n}(x^{(n)}_i)=[x^{(n)}_i]_{\overline{E}}$ for each $i\in\{1,\dots,n\}$.
Given any $\pres{\Sigma}{E}$-model $\str{A}=(A,\interp{-})$ and any function $f\colon X_n\longrightarrow A$, there exists a 
unique homomorphism of $\pres{\Sigma}{E}$-models $g\colon \str{T}^\pres{\Sigma}{E}_n\longrightarrow \str{A}$
such that $g\circ \eta_{X_n}=f$.
\[
\begin{tikzpicture}[baseline=-\the\dimexpr\fontdimen22\textfont2\relax ]
      \node (TL) at (0,1)  {$X_n$};
      \node (TR) at (3,1)  {$T^\pres{\Sigma}{E}_n$};
      \node (BR) at (3,-1) {$A$};
      \node (B) at (1.5,-1.8) {{(sets)}};
      \draw[->] (TL) to node[auto,labelsize](T) {$\eta_{X_n}$} (TR);
      \draw[->]  (TR) to node[auto,labelsize] {$g$} (BR);
      \draw[->]  (TL) to node[auto,swap,labelsize] {$f$} (BR);
\end{tikzpicture} 
\qquad
\begin{tikzpicture}[baseline=-\the\dimexpr\fontdimen22\textfont2\relax ]
      \node (TR) at (3,1)  {$\str{T}^\pres{\Sigma}{E}_n$};
      \node (BR) at (3,-1) {$\str{A}$};
      \node (B) at (3,-1.8) {($\pres{\Sigma}{E}$-models)};
      \draw[->,dashed]  (TR) to node[auto,labelsize] {$g$} (BR);
\end{tikzpicture} 
\]
\end{proposition}
\begin{proof}
The required homomorphism $g$ can be defined from $f$ by recursion; the details are omitted.
\end{proof}

\section{Clones}
\label{sec:clone}
The central notion we have introduced in the previous
section is that of 
\emph{presentation of an equational theory} 
(Definition~\ref{defn:univ_alg_pres_eq_thy}),
whose main purpose is to define its 
\emph{models} (Definition~\ref{defn:univ_alg_model}).
It can happen, however, that
two different presentations of equational theories 
define the ``same'' models, sometimes 
in a quite superficial manner.

For example, consider the following presentation of an equational theory
$\pres{\Sigma^{\group'}}{E^{\group'}}$:
\[
\Sigma^{\group'}=\Sigma^\group,
\]
\begin{multline*}
E^{\group'}_1=\{\,m(x^{(1)}_1,e)\approx x^{(1)}_1,\quad
m(e,x^{(1)}_1)\approx x^{(1)}_1, \\
m(x^{(1)}_1,i(x^{(1)}_1))\approx e,\quad
m(i(x^{(1)}_1),x^{(1)}_1)\approx e\,\},
\end{multline*}
\[
E^{\group'}_n=E^{\group}_n\quad\text{ for all }n\in\NN\setminus\{1\}.
\]
It is a classical fact that a group can be defined either 
as a model of $\pres{\Sigma^\group}{E^\group}$ or
as a model of $\pres{\Sigma^{\group'}}{E^{\group'}}$.
Indeed, we may add arbitrary {equational theorems} of 
$\pres{\Sigma^\group}{E^\group}$, such as
$i(i(x_1))\approx x_1$, $i(m(x_1,x_2))\approx m(i(x_2),i(x_1))$
and $x_1\approx x_1$,
as additional equational axioms and still obtain the groups as the models.

As another example, let us consider the presentation of an equational 
theory $\pres{\Sigma^{\group''}}{E^{\group''}}$ defined as:
\[
\Sigma^{\group''}_0=\{e,e'\}, \quad
\Sigma^{\group''}_n=\Sigma^\group_n \quad\text{ for all }n\in\NN\setminus\{0\},
\]
\[
E^{\group''}_0=\{e\approx e'\}, \quad E^{\group''}_n=E^\group_n\quad
\text{ for all }n\in\NN\setminus\{0\}.
\]
To make a set $A$ into a model of $\pres{\Sigma^{\group''}}{E^{\group''}}$,
formally we have to specify two elements $\interp{e}$ and $\interp{e'}$
of $A$, albeit they are forced to be equal and play the role of unit
with respect to the group structure determined by $\interp{m}$.
We cannot quite say that models of $\pres{\Sigma^{\group''}}{E^{\group''}}$
are \emph{equal} to models of $\pres{\Sigma^\group}{E^\group}$,
since their data differ;
however, it should be intuitively clear that there is no point
in distinguishing them.\footnote{In precise mathematical terms, our claim of the ``sameness''
amounts to the existence of an isomorphism of categories between 
the categories of $\pres{\Sigma^\group}{E^\group}$-models
and of $\pres{\Sigma^{\group''}}{E^{\group''}}$-models
preserving the underlying sets of models, i.e., 
commuting with the forgetful functors into the category $\Set$ of sets.}

\medskip

A presentation of an equational theory has much freedom in choices both of 
basic operations and of equational axioms.
It is really a \emph{presentation}.
In fact, there is a notion which may be thought of as
an \emph{equational theory} itself,
something that a presentation of an equational theory presents;
it is called an \emph{(abstract) clone} (see e.g., \cite{Taylor_clone}).

\begin{definition}
\label{def:clone}
A \defemph{clone} $\monoid{T}$ consists of:
\begin{description}
\item[(CD1)] a graded set $T=(T_n)_{n\in\NN}$;\footnote{In traditional universal algebra, people often omit $T_0$.}
\item[(CD2)] for each $n\in\NN$ and $i\in\{1,\dots,n\}$, an element
\[
p^{(n)}_i\in T_n;
\]
\item[(CD3)] for each $k,n\in\NN$, a function
\[
\circ^{(n)}_k\colon T_k\times (T_n)^k\longrightarrow T_n
\]
whose action on an element $(\phi,\theta_1,\dots,\theta_k)\in T_k\times (T_n)^k$
we write as $\phi\circ^{(n)}_k(\theta_1,\dots,\theta_k)$
or simply as $\phi\circ(\theta_1,\dots,\theta_k)$;
\end{description}
satisfying the following equations:
\begin{description}
\item[(CA1)] for each $k,n\in\NN$, $j\in\{1,\dots,k\}$ and 
$\theta_1,\dots,\theta_k\in T_n$, 
\[
p^{(k)}_j\circ^{(n)}_k(\theta_1,\dots,\theta_k) = \theta_j;
\]
\item[(CA2)] for each $n\in\NN$, $\theta\in T_n$,
\[
\theta\circ^{(n)}_n (p^{(n)}_1,\dots,p^{(n)}_n) = \theta;
\]
\item[(CA3)] for each $l,k,n\in\NN$, $\psi\in T_l$, $\phi_1,\dots,\phi_l\in
T_k$, $\theta_1,\dots,\theta_k\in T_n$, 
\begin{multline*}
\psi\circ^{(k)}_l\big(\phi_1\circ^{(n)}_k(\theta_1,\dots,\theta_k),\ 
\dots,\ \phi_l\circ^{(n)}_k(\theta_1,\dots,\theta_k)\big)
\\=
\big(\psi\circ^{(k)}_l(\phi_1,\dots,\phi_l)\big)\circ^{(n)}_k 
(\theta_1,\dots,\theta_k).
\end{multline*}
\end{description}
Such a clone is written as 
$\monoid{T}=(T,(p^{(i)}_n)_{n\in\NN,i\in\{1,\dots,n\}},
(\circ^{(n)}_k)_{k,n\in\NN})$
or simply $(T,p,\circ)$.
\end{definition}

To understand the definition of clone,
it is helpful to look at some pictures known as \emph{string diagrams}
(cf.~\cite{Curien_operad,Leinster_book}).
Given a clone $\monoid{T}=(T,p,\circ)$,
let us draw an element $\theta$ of $T_n$ as a triangle with $n$ ``{input wires}'' and a single ``{output wire}'':
\begin{equation}\label{eqn:triangle_theta}
\begin{tikzpicture}[baseline=-\the\dimexpr\fontdimen22\textfont2\relax ]
      \coordinate (TL) at (0,1);
      \coordinate (BL) at (0,-1);
      \coordinate (R) at (1.5,0);
      \path[draw,string] (TL)--(BL)--(R)--cycle;
      \node at (0.5,0) {$\theta$};
      \path[draw,string] (TL)+(0,-0.3)-- +(-0.6,-0.3);
      \node at (-0.3,0) {$\rvdots$};
      \path[draw,string] (BL)+(0,0.3)-- +(-0.6,0.3);
      \draw [decorate,decoration={brace,amplitude=5pt,mirror}]
      (TL)++(-0.8,-0.1) -- +(0,-1.8) node [midway,xshift=-0.5cm,labelsize] 
      {$n$};
      \path[draw,string] (R)-- +(0.6,0);
\end{tikzpicture}
\end{equation}
The element $p^{(n)}_i$ in (CD2) may also be denoted by
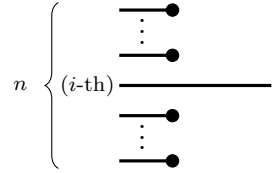
\begin{equation*}
\label{eqn:clone_p}
\begin{tikzpicture}[baseline=-\the\dimexpr\fontdimen22\textfont2\relax ]
      \coordinate (t) at (0,1);
      \coordinate (s) at (0,0.4);
      \coordinate (t3) at (0,0);
      \coordinate (t4) at (0,-0.4);
      \coordinate (t5) at (0,-1);
      \node [circle,draw,inner sep=0.15em,fill] at (0.7,-1) {};
      \node [circle,draw,inner sep=0.15em,fill] at (0.7,-0.4) {};  
      \node [circle,draw,inner sep=0.15em,fill] at (0.7,0.4) {};
      \node [circle,draw,inner sep=0.15em,fill] at (0.7,1) {};
      \draw [string] (0,1) -- (0.7,1);
      \node at (0.3,0.7) {$\rvdots$};
      \draw [string] (0,0.4) -- (0.7,0.4);
      \draw [string] (0,0) -- (2,0);
      \draw [string] (0,-1) -- (0.7,-1);
      \node at (0.3,-0.7) {$\rvdots$};
      \draw [string] (0,-0.4) -- (0.7,-0.4);
      \node [labelsize, left of=t3, node distance=0.42cm] {$(i\text{-th})$};
      \draw [decorate,decoration={brace,amplitude=5pt,mirror}]
      (-0.8,1.1) -- +(0,-2.2) node [midway,xshift=-0.5cm,labelsize] 
      {$n$};
\end{tikzpicture}
\end{equation*}
and $\phi\circ^{(n)}_k(\theta_1,\dots,\theta_k)$ in (CD3) by
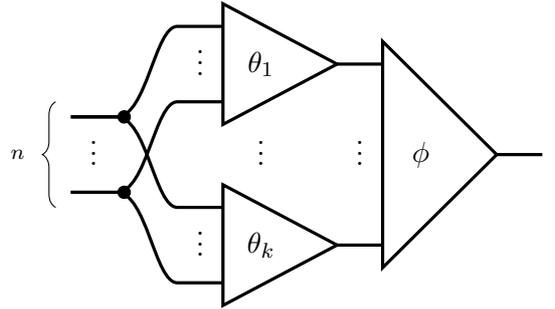
\begin{equation*}
\label{eqn:clone_composition}
\begin{tikzpicture} [baseline={([yshift=-0.5ex]current bounding box.center)}]
      \coordinate (TL) at (0,1.5);
      \coordinate (BL) at (0,-1.5);
      \coordinate (R) at (1.5,0);
      \path[draw,string] (TL)--(BL)--(R)--cycle;
      \node at (0.5,0) {$\phi$};
      \path[draw,string] (TL)+(0,-0.3)-- +(-0.6,-0.3);
      \node at (-0.3,0) {$\rvdots$};
      \path[draw,string] (BL)+(0,0.3)-- +(-0.6,0.3);
      \path[draw,string] (R)-- +(0.6,0);
  \begin{scope}[shift={(-2.1,1.2)}]
      \coordinate (TL) at (0,0.8);
      \coordinate (BL) at (0,-0.8);
      \coordinate (R) at (1.5,0);
      \path[draw,string] (TL)--(BL)--(R)--cycle;
      \node at (0.5,0) {$\theta_1$};
      \node at (0.5,-1.2) {$\rvdots$};
      \node at (-0.3,0) {$\rvdots$};
  \end{scope}
  \begin{scope}[shift={(-2.1,-1.2)}]
      \coordinate (TL) at (0,0.8);
      \coordinate (BL) at (0,-0.8);
      \coordinate (R) at (1.5,0);
      \path[draw,string] (TL)--(BL)--(R)--cycle;
      \node at (0.5,0) {$\theta_{k}$};
      \node at (-0.3,0) {$\rvdots$};
  \end{scope}
  \begin{scope}[shift={(-2.1,0)}]
      \node [circle,draw,inner sep=0.15em,fill] at (-1.3,0.5) {};
      \node [circle,draw,inner sep=0.15em,fill] at (-1.3,-0.5) {};
      \draw [string] (-2,0.5) -- (-1.4,0.5) ..controls (-1,0.5) and  (-0.9,1.7) 
      .. (-0.6,1.7)--(0,1.7);
      \draw [string] (-2,0.5) -- (-1.4,0.5) ..controls (-1,0.5) and  
      (-0.9,-0.7) .. (-0.6,-0.7)--(0,-0.7);
      \draw [string] (-2,-0.5) -- (-1.4,-0.5) ..controls (-1,-0.5) and  
      (-0.9,0.7) .. (-0.6,0.7)--(0,0.7);
      \draw [string] (-2,-0.5) -- (-1.4,-0.5) ..controls (-1,-0.5) and  
      (-0.9,-1.7) .. (-0.6,-1.7)--(0,-1.7);
      \node at (-1.7,0) {$\rvdots$};
      \draw [decorate,decoration={brace,amplitude=5pt}]
      (-2.2,-0.7) -- (-2.2,0.7) node [midway,xshift=-0.5cm,labelsize] 
      {$n$};
  \end{scope}
\end{tikzpicture}\ .
\end{equation*}
Then the axioms (CA1)--(CA3) simply assert obvious equations 
between the resulting ``circuits''.
For instance, (CA2) for $n=3$ reads:
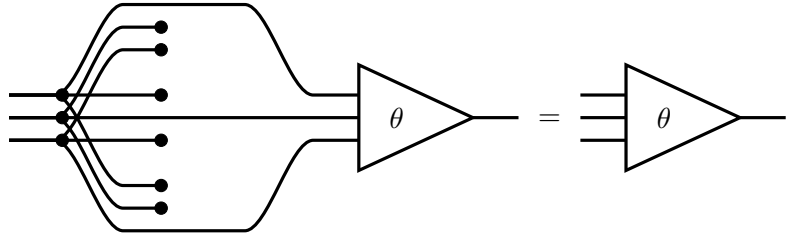
\begin{equation*}
\begin{tikzpicture}[baseline=-\the\dimexpr\fontdimen22\textfont2\relax ]
      \draw [string] (0,0.3) -- (0.6,0.3) ..controls (0.9,0.3) and  (1.2,1.5) 
      .. (1.5,1.5)--(3.1,1.5)..controls (3.4,1.5) and  (3.7,0.3) .. 
      (4,0.3)--(4.6,0.3);
      \draw [string] (0,0) -- (0.6,0) ..controls (0.9,0) and  (1.2,1.2) .. 
      (1.5,1.2)--(2,1.2);
      \draw [string] (0,-0.3) -- (0.6,-0.3) ..controls (0.9,-0.3) and  
      (1.2,0.9) .. (1.5,0.9)--(2,0.9);
      \draw [string] (0,0.3) -- (2,0.3);
      \draw [string] (0,0) -- (4.6,0);
      \draw [string] (0,-0.3) -- (2,-0.3);
      \draw [string] (0,-0.3) -- (0.6,-0.3) ..controls (0.9,-0.3) and  
      (1.2,-1.5) .. (1.5,-1.5)--(3.1,-1.5)..controls (3.4,-1.5) and  (3.7,-0.3) 
      .. (4,-0.3)--(4.6,-0.3);
      \draw [string] (0,0) -- (0.6,0) ..controls (0.9,0) and  (1.2,-1.2) .. 
      (1.5,-1.2)--(2,-1.2);
      \draw [string] (0,0.3) -- (0.6,0.3) ..controls (0.9,0.3) and  (1.2,-0.9) 
      .. (1.5,-0.9)--(2,-0.9);
      \node [circle,draw,inner sep=0.15em,fill] at (2,1.2) {};
      \node [circle,draw,inner sep=0.15em,fill] at (2,0.9) {};  
      \node [circle,draw,inner sep=0.15em,fill] at (2,0.3) {};
      \node [circle,draw,inner sep=0.15em,fill] at (2,-0.3) {};  
      \node [circle,draw,inner sep=0.15em,fill] at (2,-0.9) {};
      \node [circle,draw,inner sep=0.15em,fill] at (2,-1.2) {}; 
      \node [circle,draw,inner sep=0.15em,fill] at (0.7,0.3) {};
      \node [circle,draw,inner sep=0.15em,fill] at (0.7,0) {};
      \node [circle,draw,inner sep=0.15em,fill] at (0.7,-0.3) {};
      \path[draw,string] (4.6,0.7)--(4.6,-0.7)--(6.1,0)--cycle;
      \node at (5.1,0) {$\theta$};
      \draw [string] (6.1,0) -- (6.7,0);
\end{tikzpicture}
\ =\ 
\begin{tikzpicture}[baseline=-\the\dimexpr\fontdimen22\textfont2\relax ]
      \draw [string] (0,0.3) -- (0.6,0.3);
      \draw [string] (0,0) -- (0.6,0);
      \draw [string] (0,-0.3) -- (0.6,-0.3);

      \path[draw,string] (0.6,0.7)--(0.6,-0.7)--(2.1,0)--cycle;
      \node at (1.1,0) {$\theta$};
      \draw [string] (2.1,0) -- (2.7,0);
\end{tikzpicture}\ .
\end{equation*}

Next we define {models} of a clone.
We first need a few preliminary definitions.

\begin{definition}
\label{def:endo_clone}
Let $A$ be a set.
Define the clone $\Endcl{A}=(\enrich{A}{A},p,\circ)$
as follows:
\begin{description}
\item[(CD1)] for each $n\in\NN$, let $\enrich{A}{A}_n$ be the set of all
functions from $A^n$ to $A$; 
\item[(CD2)] for each $n\in\NN$ and $i\in\{1,\dots,n\}$,
let $p^{(n)}_i$ be the $i$-th projection 
$A^n\longrightarrow A$, $(a_1,\dots,a_n)\longmapsto a_i$;
\item[(CD3)] for each $k,n\in\NN$, $g\colon A^k\longrightarrow A$ and
$f_1,\dots,f_k\colon A^n\longrightarrow A$, 
let $g\circ^{(n)}_k(f_1,\dots,f_k)$ be the function
$(a_1,\dots,a_n)\longmapsto g(f_1(a_1,\dots,a_n),\dots,f_k(a_1,\dots,a_n))$,
that is, the following composite:
\[
\begin{tikzpicture}[baseline=-\the\dimexpr\fontdimen22\textfont2\relax ]
      \node (1) at (0,0)  {$A^n$};
      \node (2) at (3,0)  {$A^k$};
      \node (3) at (5,0) {$A.$};
      \draw[->] (1) to node[auto,labelsize]{$\langle f_1,\dots,f_k\rangle$} (2);
      \draw[->] (2) to node[auto,labelsize] {$g$} (3);
\end{tikzpicture} 
\]
\end{description}
It is straightforward to check the axioms (CA1)--(CA3).
\end{definition}

\begin{definition}
\label{def:clone_hom}
Let $\monoid{T}=(T,p,\circ)$ and $\monoid{T'}=(T',p',\circ')$ be 
clones.
A \defemph{clone homomorphism from $\monoid{T}$ to $\monoid{T'}$}
is a morphism of graded sets (Definition~\ref{def:graded_set}) 
$h\colon T\longrightarrow T'$ which preserves the structure of clones;
precisely,
\begin{itemize}
\item for each $n\in\NN$ and $i\in\{1,\dots,n\}$, $h_n(p^{(n)}_i)=p'^{(n)}_i$;
\item for each $k,n\in\NN$, $\phi\in T_k$ and $\theta_1,\dots,\theta_k\in T_n$,
\[
h_n\big(\phi\circ^{(n)}_k(\theta_1,\dots,\theta_k)\big)=
h_k(\phi)\circ'^{(n)}_k \big(h_n(\theta_1),\dots,h_n(\theta_k)\big).\qedhere
\]
\end{itemize}
\end{definition}

\begin{definition}
\label{def:clone_model}
Let $\monoid{T}$ be a clone.
A \defemph{model of $\monoid{T}$} is a pair $\str{A}=(A,\alpha)$ consisting of a set $A$ and 
a clone homomorphism $\alpha\colon \monoid{T}\longrightarrow \Endcl{A}$.
\end{definition}

Let us then define the notion of homomorphism between models.
First we extend the definition of the graded set $\enrich{A}{A}$ 
introduced in Definition~\ref{def:endo_clone}.

\begin{definition}
\begin{enumerate}
\item Let $A$ and $B$ be sets. The graded set $\enrich{A}{B}$ is defined by 
setting,
for each $n\in\NN$, $\enrich{A}{B}_n$ be the set of all functions from $A^n$
to $B$.
\item Let $A,A'$ and $B$ be sets and $f\colon A'\longrightarrow A$ be a 
function.
The morphism of graded sets 
$\enrich{f}{B}\colon\enrich{A}{B}\longrightarrow\enrich{A'}{B}$
is defined by setting, for each $n\in\NN$, 
$\enrich{f}{B}_n\colon\enrich{A}{B}_n\longrightarrow\enrich{A'}{B}_n$ be the 
precomposition by $f^n\colon (A')^n\longrightarrow A^n$;
that is, $h\longmapsto h\circ f^n$. 
\item Let $A,B$ and $B'$ be sets and $g\colon B\longrightarrow B'$ be a 
function.
The morphism of graded sets 
$\enrich{A}{g}\colon\enrich{A}{B}\longrightarrow\enrich{A}{B'}$
is defined by setting, for each $n\in\NN$, 
$\enrich{A}{g}_n\colon\enrich{A}{B}_n\longrightarrow\enrich{A}{B'}_n$ be the 
postcomposition by $g\colon B\longrightarrow B'$;
that is, $h\longmapsto g\circ h$.\qedhere
\end{enumerate}
\end{definition}

\begin{definition}
\label{def:clone_mod_hom}
Let $\monoid{T}$ be a clone, and $\str{A}=(A,\alpha)$ and $\str{B}=(B,\beta)$ be models of 
$\monoid{T}$.
A \defemph{homomorphism from $\str{A}$ to $\str{B}$} is a function 
$f\colon A\longrightarrow B$ making the following diagram of 
morphisms of graded sets commute:
\[
\begin{tikzpicture}[baseline=-\the\dimexpr\fontdimen22\textfont2\relax ]
      \node (TL) at (0,2)  {$T$};
      \node (TR) at (3,2)  {$\enrich{A}{A}$};
      \node (BL) at (0,0) {$\enrich{B}{B}$};
      \node (BR) at (3,0) {$\enrich{A}{B}.$};
      \draw[->] (TL) to node[auto,labelsize](T) {$\alpha$} (TR);
      \draw[->]  (TR) to node[auto,labelsize] {$\enrich{A}{f}$} (BR);
      \draw[->]  (TL) to node[auto,swap,labelsize] {$\beta$} (BL);
      \draw[->] (BL) to node[auto,swap,labelsize](B) {$\enrich{f}{B}$} (BR);
\end{tikzpicture} \qedhere
\]
\end{definition}

\medskip

Now let us turn to the relation between
presentations of equational theories 
(Definition~\ref{defn:univ_alg_pres_eq_thy})
and clones.
We start with the observation that the graded set $T(\Sigma)$ of $\Sigma$-terms
(Definition~\ref{def:Sigma_term})
has a canonical clone structure, given as follows:
\begin{description}
\item[(CD2)] for each $n\in\NN$ and $i\in\{1,\dots,n\}$, let $p^{(n)}_i$ be 
$x^{(n)}_i\in T(\Sigma)_n$;
\item[(CD3)] for each $k,n\in\NN$, $s\in T(\Sigma)_k$ and $t_1,\dots,t_k\in 
T(\Sigma)_n$, let $s\circ^{(n)}_k(t_1,\dots,t_k)$ be
$s[x^{(k)}_1\mapsto t_1,\dots,x^{(k)}_k\mapsto t_k]\in T(\Sigma)_n$.
\end{description}
We denote the resulting clone by $\monoid{T}(\Sigma)$.
In fact, this clone is characterised as the \emph{free clone generated by $\Sigma$}, in the following sense.

\begin{proposition}
\label{prop:free_clone}
Let $\Sigma$ be a graded set, and 
let $\eta_\Sigma\colon \Sigma\longrightarrow
T(\Sigma)$ be the morphism of graded sets defined by 
$(\eta_\Sigma)_n(\sigma)=\sigma(x^{(n)}_1,\dots,x^{(n)}_n)$ for each $n\in\NN$ 
and $\sigma\in\Sigma_n$.
Given any clone $\monoid{S}=(S,p,\circ)$ and any morphism of 
graded sets $f\colon \Sigma\longrightarrow S$, there exists a 
unique clone homomorphism $g\colon \monoid{T}(\Sigma)\longrightarrow \monoid{S}$
such that $g\circ \eta_\Sigma=f$.
\[
\begin{tikzpicture}[baseline=-\the\dimexpr\fontdimen22\textfont2\relax ]
      \node (TL) at (0,1)  {$\Sigma$};
      \node (TR) at (3,1)  {${T}(\Sigma)$};
      \node (BR) at (3,-1) {$S$};
      \node (B) at (1.5,-1.8) {{(graded sets)}};
      \draw[->] (TL) to node[auto,labelsize](T) {$\eta_\Sigma$} (TR);
      \draw[->]  (TR) to node[auto,labelsize] {$g$} (BR);
      \draw[->]  (TL) to node[auto,swap,labelsize] {$f$} (BR);
\end{tikzpicture} 
\qquad
\begin{tikzpicture}[baseline=-\the\dimexpr\fontdimen22\textfont2\relax ]
      \node (TR) at (3,1)  {$\monoid{T}(\Sigma)$};
      \node (BR) at (3,-1) {$\monoid{S}$};
      \node (B) at (3,-1.8) {(clones)};
      \draw[->,dashed]  (TR) to node[auto,labelsize] {$g$} (BR);
\end{tikzpicture} 
\]
\end{proposition}
\begin{proof}
The clone homomorphism $g$ may be defined by recursion (recall that $T(\Sigma)$ was defined inductively) as follows:
\begin{enumerate}
\item for each $n\in\NN$ and $i\in\{1,\dots,n\}$, let
\[
g_n(x^{(n)}_i)=p^{(n)}_i;
\]
\item for each $k,n\in\NN$, $\sigma\in\Sigma_k$ and $t_1,\dots,t_k\in 
T(\Sigma)_n$, let 
\[
g_n(\sigma(t_1,\dots,t_k))=f_k(\sigma)\circ^{(n)}_k(g_n(t_1),\dots,g_n(t_k)).
\]
\end{enumerate}
To check that $g$ is indeed a clone homomorphism, it suffices to show 
for each $s\in T(\Sigma)_k$ and $t_1,\dots,t_k\in T(\Sigma)_n$,
\[
g_n(s[x^{(k)}_1\mapsto t_1,\dots,x^{(k)}_k\mapsto t_k])=
g_k(s)\circ^{(n)}_k(g_n(t_1),\dots, g_n(t_k));
\]
this can be shown by induction on $s$.
The uniqueness of $g$ is clear.
\end{proof}
The construction given in Definition~\ref{def:interp_Sigma_term}
is a special case of the above; let $\monoid{S}$ be $\Endcl{A}$.

Recall from Definition~\ref{def:eq_logic} the 
graded set $\overline{E}\subseteq T(\Sigma)\times T(\Sigma)$ 
of equational theorems of a presentation of an equational theory 
$\pres{\Sigma}{E}$.
The graded set $\overline{E}$ is an equivalence relation on $T(\Sigma)$, and hence we may consider the quotient graded set $T(\Sigma)/\overline{E}$ (as we did in the proof of Theorem~\ref{thm:eq_logic_sound_complete}).
By the rule ({\sc{Cong}}), the clone operations on $T(\Sigma)$
induce well-defined operations on $T(\Sigma)/\overline{E}$;
that is, $\overline{E}$ is not only an equivalence relation on the graded set $T(\Sigma)$, but it is also a congruence relation on the clone $\monoid{T}(\Sigma)$.
In particular, we can define $\circ^{(n)}_k$ on $T(\Sigma)/\overline{E}$ by
\[
[\phi]_{\overline{E}}\circ^{(n)}_k([\theta_1]_{\overline{E}},\dots,
[\theta_k]_{\overline{E}})= 
[\phi(\theta_1,\dots,\theta_k)]_{\overline{E}}.
\]
This makes the graded set $T(\Sigma)/\overline{E}$ into a clone;
the clone axioms for $T(\Sigma)/{\overline{E}}$ may be immediately checked from the existence of a surjective morphism of graded sets 
$q\colon T(\Sigma)\longrightarrow T(\Sigma)/\overline{E}$ (given by 
$\theta\longmapsto [\theta]_{\overline{E}}$) preserving 
the clone operations.
The resulting clone is denoted by $\monoid{T}^\pres{\Sigma}{E}$;
in words, it is the clone consisting of \emph{$\Sigma$-terms modulo equational theorems of $\pres{\Sigma}{E}$}.
It is also characterised by a universal property.
\begin{proposition}
\label{prop:quotient_clone}
Let $\pres{\Sigma}{E}$ be a presentation of an equational theory,
and let  $q\colon \monoid{T}(\Sigma)\longrightarrow 
\monoid{T}^\pres{\Sigma}{E}$ be the clone homomorphism defined by $q_n(\theta)=[\theta]_{\overline{E}}$
for each $n\in\NN$ and $\theta\in T(\Sigma)_n$.
Given any clone $\monoid{S}=(S,p,\circ)$ and a clone homomorphism 
$g\colon \monoid{T}(\Sigma)\longrightarrow \monoid{S}$ such that for any 
$t\approx_n s\in E$, $g_n(t)=g_n(s)$ holds, there exists a unique
clone homomorphism $h\colon 
\monoid{T}^\pres{\Sigma}{E}\longrightarrow\monoid{S}$
such that $h\circ q=g$.
\[
\begin{tikzpicture}[baseline=-\the\dimexpr\fontdimen22\textfont2\relax ]
      \node (TL) at (0,1)  {$\monoid{T}(\Sigma)$};
      \node (TR) at (3,1)  {$\monoid{T}^\pres{\Sigma}{E}$};
      \node (BR) at (3,-1) {$\monoid{S}$};
      \draw[->] (TL) to node[auto,labelsize](T) {$q$} (TR);
      \draw[->,dashed]  (TR) to node[auto,labelsize] {$h$} (BR);
      \draw[->]  (TL) to node[auto,swap,labelsize] {$g$} (BR);
\end{tikzpicture} 
\]
\end{proposition}
\begin{proof}
The clone homomorphism $h$ is given by 
$h_n([\theta]_{\overline{E}})=g_n(\theta)$;
this is shown to be well-defined by induction on $\pres{\Sigma}{E}$-proofs
(see Definition~\ref{def:eq_logic}).
The uniqueness of $h$ is immediate from the surjectivity of $q$.
\end{proof}

We can now show that for any presentation of an equational theory 
$\pres{\Sigma}{E}$,
to give a model of $\pres{\Sigma}{E}$ is equivalent to give a model of the clone
$\monoid{T}^\pres{\Sigma}{E}$.
A model of the clone $\monoid{T}^\pres{\Sigma}{E}$ 
(Definition~\ref{def:clone_model}) can be---by 
Proposition~\ref{prop:quotient_clone}---equivalently given as a suitable clone 
homomorphism out of 
$\monoid{T}(\Sigma)$; this in turn is---by 
Proposition~\ref{prop:free_clone}---equivalently given as a suitable morphism 
of graded sets 
out of $\Sigma$,
which is nothing but a model of the presentation of an equational theory 
$\pres{\Sigma}{E}$ (Definition~\ref{defn:univ_alg_model}).

We also remark that every clone is isomorphic to a clone of the form 
$\monoid{T}^\pres{\Sigma}{E}$
for some presentation of an equational theory $\pres{\Sigma}{E}$.
Indeed, given any clone $\monoid{S}=(S,p,\circ)$ we can consider its underlying graded set $S$ as a graded set of basic operations, and obtain the surjective clone homomorphism $\varepsilon_\monoid{S}\colon \monoid{T}(S)\longrightarrow \monoid{S}$ extending the identity morphism on $S$ by Proposition~\ref{prop:free_clone}. 
Define $E_\monoid{S}\subseteq T(S)\times T(S)$ to be the kernel of $\varepsilon_\monoid{S}$, i.e., the graded set of all pairs of elements of $T(S)$ whose images under $\varepsilon_\monoid{S}$ agree.
Then we have $\monoid{S}\cong\monoid{T}^\pres{S}{E_\monoid{S}}$.

\medskip

The inference rules of equational logic we have given in
Definition~\ref{def:eq_logic} can be understood as 
the inductive definition of the congruence relation 
$\overline{E}\subseteq T(\Sigma)\times T(\Sigma)$  
on the clone $\monoid{T}(\Sigma)$
generated by $E\subseteq T(\Sigma)\times T(\Sigma)$.
The notion of clone therefore provides conceptual understanding 
of equational logic. 

We can also shed new light on the soundness and completeness theorem 
(Theorem~\ref{thm:eq_logic_sound_complete}) for equational logic.
First we define a variant of the semantical consequence relation $\vDash$ (Definition~\ref{def:semantical_consequence_rel}) via the ``clone-valued semantics''.

\begin{definition}
\label{def:abstract_semantical_consequence_rel}
\begin{enumerate}
\item Let $\Sigma$ be a graded set, $\monoid{S}=(S,p,\circ)$ be a clone 
and $f\colon \Sigma\longrightarrow S$ be a morphism of graded set.
For any $\Sigma$-equation $t\approx_n s\in T(\Sigma)\times T(\Sigma)$,
we write 
\[
(\monoid{S},f)\vDash_\clo t\approx_n s 
\]
iff $g(t)=g(s)$,
where $g\colon \monoid{T}(\Sigma)\longrightarrow \monoid{S}$
is the clone homomorphism extending $f$ via Proposition~\ref{prop:free_clone}.
\item Let $\pres{\Sigma}{E}$ be a presentation of an equational theory.
For any $\Sigma$-equation $t\approx_n s\in T(\Sigma)\times T(\Sigma)$,
we write 
\[
\pres{\Sigma}{E}\vDash_\clo t\approx_n s 
\]
iff for any clone $\monoid{S}=(S,p,\circ)$ and a morphism of graded set 
$f\colon \Sigma\longrightarrow S$ such that 
$(\monoid{S},f)\vDash_\clo t'\approx_{n'} s'$ for all
$t'\approx_{n'} s'\in E$,
$(\monoid{S},f)\vDash_\clo t\approx_n s $.\qedhere
\end{enumerate}
\end{definition} 
\begin{thm}[cf.~Theorem~\ref{thm:eq_logic_sound_complete}]
\label{thm:abst_sound_complete}
Let $\pres{\Sigma}{E}$ be a presentation of an equational theory.
\begin{enumerate}
\item (Soundness with respect to the clone-valued semantics) Let $t\approx_n s\in T(\Sigma)\times T(\Sigma)$.
If $\pres{\Sigma}{E}\vdash t\approx_n s$ then $\pres{\Sigma}{E}\vDash_\clo 
t\approx_n s$.
\item (Completeness with respect to the clone-valued semantics) Let $t\approx_n s\in T(\Sigma)\times T(\Sigma)$.
If $\pres{\Sigma}{E}\vDash_\clo t\approx_n s$ then $\pres{\Sigma}{E}\vdash 
t\approx_n s$.
\end{enumerate}
\end{thm}
\begin{proof}
The soundness theorem with respect to the clone-valued semantics follows from 
Proposition~\ref{prop:quotient_clone}.
For the completeness theorem with respect to the clone-valued semantics, consider the 
clone $\monoid{T}^\pres{\Sigma}{E}$ and the morphism of graded set
\[
\begin{tikzpicture}[baseline=-\the\dimexpr\fontdimen22\textfont2\relax ]
      \node (1) at (0,0)  {$\Sigma$};
      \node (2) at (2,0)  {$T(\Sigma)$};
      \node (3) at (4.5,0) {$T(\Sigma)/\overline{E}$};
      \draw[->] (1) to node[auto,labelsize]{$\eta_\Sigma$} (2);
      \draw[->] (2) to node[auto,labelsize] {$q$} (3);
\end{tikzpicture} 
\]
(see Proposition~\ref{prop:free_clone} for the definition of $\eta_\Sigma$
and Proposition~\ref{prop:quotient_clone} for $q$);
then $(\monoid{T}^\pres{\Sigma}{E},q\circ \eta_\Sigma)\vDash_\clo t\approx_n s$
iff $\pres{\Sigma}{E}\vdash t\approx_n s$.
\end{proof}
Clearly, $\pres{\Sigma}{E}\vDash_\clo t\approx_n s$ 
implies $\pres{\Sigma}{E}\vDash t\approx_n s$;
the latter amounts to restricting the clone $\monoid{S}$
in Definition~\ref{def:abstract_semantical_consequence_rel}
to those of the form $\Endcl{A}$ for some set $A$.
Hence the (original) soundness theorem follows from the soundness theorem with respect to the clone-valued semantics,
but observe that the completeness theorem is not an immediate consequence of 
the completeness theorem with respect to the clone-valued semantics.\footnote{However, one can combine the completeness theorem with respect to the clone-valued semantics with an \emph{embedding theorem for clones}, which claims that every clone $\monoid{S}$ can be embedded into a product of clones of the form $\End(A)$, to obtain the completeness theorem. Such an embedding may be obtained, for example, by canonically mapping an arbitrary clone $\monoid{S}=(S,p,\circ)$ into $\prod_{i=0}^\infty\End{(S_n)}$, whose injectivity can be checked by an argument similar to our proof of the completeness theorem (Theorem~\ref{thm:eq_logic_sound_complete}).}

\bibliographystyle{alpha} %
\bibliography{myref} %

\end{document}